\documentclass[11pt]{amsart}

\usepackage{amsfonts}
\usepackage{amssymb}
\usepackage{amsmath}
\usepackage{amsthm}
\usepackage{color}
\usepackage{graphicx}

\usepackage[bookmarksnumbered, colorlinks, plainpages]{hyperref}
\hypersetup{colorlinks=true,linkcolor=red, anchorcolor=green, citecolor=cyan, urlcolor=red, filecolor=magenta, pdftoolbar=true}

\newtheorem{theorem}{Theorem}[section]
\newtheorem{lemma}[theorem]{Lemma}

\newtheorem{corollary}[theorem]{Corollary}
\theoremstyle{definition}
\newtheorem{definition}[theorem]{Definition}

\theoremstyle{remark}

\numberwithin{equation}{section}

\def\ds{\displaystyle}

\numberwithin{equation}{section}

\newcommand{\RR}{{\mathbb R}}
\newcommand{\NN}{{\mathbb N}}

\newcommand{\MS}{{\mathcal M}}
\newcommand{\EE}{{\mathbb E}}

\providecommand{\norm}[1]{\left\rVert#1\right\rVert}

\begin{document}

\setcounter{page}{1}

\title[Radon-Nikodym property for vector valued measures]{A characterization of the Radon-Nikodym property for vector valued measures}

\author[P. Mikusi\'nski \MakeLowercase{and} J.P. Ward]{Piotr Mikusi\'nski$^1$  \MakeLowercase{and}  John Paul Ward,$^{2}$}

\address{$^{1}$Department of Mathematics, University of Central Florida, Orlando, Florida.}
\email{\textcolor[rgb]{0.00,0.00,0.84}{Piotr.Mikusinski@ucf.edu}}

\address{$^{2}$Department of Mathematics, University of Central Florida, Orlando, Florida. 
Department of Mathematics, North Carolina A\&T State University, Greensboro, North Carolina.}
\email{\textcolor[rgb]{0.00,0.00,0.84}{jpward@ncat.edu}}



%\date{Received: xxxxxx; Revised: yyyyyy; Accepted: zzzzzz.
%\newline \indent $^{*}$Corresponding author}

\begin{abstract}  If $\mu_1,\mu_2,\dots$ are positive measures on a measurable space $(X,\Sigma)$ and $v_1,v_2, \dots$ are elements of a Banach space $\EE$ such that $\sum_{n=1}^\infty \|v_n\| \mu_n(X) < \infty$, then $\omega (S)= \sum_{n=1}^\infty v_n \mu_n(S)$
defines a vector measure of bounded variation on $(X,\Sigma)$. We show $\EE$ has the Radon-Nikodym property if and only if every $\EE$-valued measure of bounded variation on $(X,\Sigma)$ is of this form.

As an application of this result we show that under natural conditions an operator defined on positive measures, has a unique extension to an operator defined on $\EE$-valued measures for any Banach space $\EE$ that has the Radon-Nikodym property.
\end{abstract}

\maketitle

\section{Introduction}

A Banach space $\EE$ has the {\it Radon-Nikodym property} with respect to a measure space $(X,\Sigma, \mu)$ if for every $\EE$-valued measure $\omega$ on $(X,\Sigma)$ of bounded variation that is absolutely continuous with respect $\mu$ there exists a Bochner integrable function $f$ on $(X,\Sigma, \mu)$ such that $\omega (S)= \int_S f d\mu$ for every $S\in\Sigma$. We say that $\EE$ has the {\it Radon-Nikodym property} if $\EE$ has Radon-Nikodym property with respect to every finite measure space. Not every Banach space has the Radon-Nikodym property (see \cite{DiestelUhl}).

If $\mu_1,\mu_2,\dots$ are positive measures on a measurable space $(X,\Sigma)$ and $v_1,v_2, \dots \in \EE$ are such that $\sum_{n=1}^\infty \|v_n\| \mu_n(X) < \infty$,
then $\omega (S)= \sum_{n=1}^\infty v_n \mu_n(S)$ defines a vector measure of bounded variation on $(X,\Sigma)$. We will show that every $\EE$-valued measure of bounded variation is of this form if and only if $\EE$ has the Radon-Nikodym property. 
 
As an example of application of the above result we will prove an extension theorem for transfunctions. By a transfunction (see \cite{trans} and \cite{JBPM1}) we mean a map between sets of measures on measurable spaces.  More precisely, if $(X,\Sigma_X)$ and $(Y, \Sigma_Y)$ are measurable spaces and  $\MS(X,\Sigma_X, \RR^+)$ and $\MS(Y,\Sigma_Y, \RR^+)$ are the sets of finite positive measures on $\Sigma_X$ and on $\Sigma_Y$, respectively, by a {\it transfunction} from $(X,\Sigma_X)$ to $(Y, \Sigma_Y)$ we mean a map $\Phi : \MS(X,\Sigma_X, \RR^+) \to \MS(Y,\Sigma_Y, \RR^+)$. If $f:(X,\Sigma_X)\to(Y, \Sigma_Y)$ is a measurable function, then
$\Phi_f(\mu)(B)= \mu(f^{-1}(B))$,  
for $\mu\in \MS(X,\Sigma_X, \RR^+)$ and $B\in \Sigma_Y$, defines a transfunction from $\MS(X,\Sigma_X, \RR^+)$ to $\MS(Y,\Sigma_Y, \RR^+)$. Properties of transfunctions related to functions are discussed in \cite{JBPM1}.
 
We are interested in transfunctions as a generalization of a function from $X$ to $Y$.  Instead of mapping a point $x\in X$ to a point $y\in Y$ a transfunction can be thought of as mapping a probability distribution of the input to a probability distribution of the output. Another interpretation of a transfunction could be a change in a population distributed in $X$. The total population can increase or decrease and its distribution in $X$ can change. A transfunction captures all these changes.

The definition of transfunctions makes sense if finite positive measures are replaced by vector valued measures of bounded variation. In the last section of this note we define extensions of transfunctions to vector measures. We also discuss the question of uniqueness of such extensions.

We have two reasons to consider extensions of transfunctions to signed measures and vector measures. First, for some applications it is more natural or even necessary to use or vector measures. Second, by extending the domain of a transfunction to a vector space we are able to use tools from functional analysis.

\section{Measures with values in a Banach space with the Radon-Nikodym property}

In this note we use the same symbol to denote a subset of $X$ and the characteristic function of that set, that is, if $A\subset X$ we will write
$$
A(x) = \left\{ \begin{array}{ll}
1 & \mbox{if } x\in A\\
0 & \mbox{otherwise}\\
\end{array} \right. .
$$
By $\mu_A$ we denote the restriction of $\mu$ to $A$, that is, $\mu_A(S)=\mu(S\cap A)$.

Let $X$ be a nonempty set, $\Sigma$ a $\sigma$-algebra of subsets of $X$, and let $(\EE,\|\cdot\|)$ be a Banach space. By the {\it variation} of a set function $\mu : \Sigma \to \EE$ we mean the set function $|\mu| : \Sigma \to [0,\infty]$ defined by
$$
|\mu|(A)= \sup \left\{ \sum_{B\in \pi} ||\mu (B)|| : \pi\subset \Sigma \text{ is a finite partition of } A \right\}.
$$
Note that $|\mu|(A) \geq \|\mu(A)\|$ for any set $A \in \Sigma$. If $|\mu|(X)<\infty$, then we say that $\mu$ is of {\it bounded variation}. A $\sigma$-additive set function of bounded variation will be called an $\EE$-valued measure or simply a vector measure.

Let $(X,\Sigma, \mu)$ be a measure space and let $\EE$ be a Banach space. We use the following definition of Bochner integrable functions (see \cite{TheBochnerIntegral} or \cite{integrals}):
\begin{definition}
A function $f:X\to \EE$ is called Bochner integrable  if there are sets $A_1, A_2,\dots \in \Sigma$ and vectors $v_1, v_2, \dots \in \EE$ such that 
$\sum_{n=1}^\infty \mu(A_n)<\infty$ and
$f(x)=\sum_{n=1}^\infty v_n A_n(x)$ for every $x\in X$ for which $\sum_{n=1}^\infty \|v_n\|A_n(x) < \infty$.

If for some $f:X\to \EE$, $A_1, A_2, \dots \in \Sigma$, and $v_1, v_2, \dots \in \EE$, both conditions are satisfied, we write $f \simeq \sum_{n=1}^\infty v_n A_n$. If
$f \simeq \sum_{n=1}^\infty v_n A_n$, then we define
$\int f d\mu = \sum_{n=1}^\infty v_n \mu(A_n)$.
\end{definition}

If $f:X\to \EE$ is a Bochner integrable function on a measure space $(X,\Sigma, \mu)$, then
$\omega (S)= \int_S f d\mu$
defines a vector measure of bounded variation on $(X,\Sigma)$. Not every vector measure is of this form (see \cite{DiestelUhl}).

\begin{theorem}\label{Th4.1}
A Banach space $\EE$ has the Radon-Nikodym property with respect to a measure space $(X,\Sigma, \mu)$ if and only if every $\EE$-valued measure $\omega$ on $(X,\Sigma)$ of bounded variation that is absolutely continuous with respect $\mu$ is of the form
\[
\omega (S)= \sum_{n=1}^\infty v_n \mu_n(S)
\]
where $\mu_1,\mu_2, \dots$ are positive measures on $(X,\Sigma)$ that are absolutely continuous with respect $\mu$ and $v_1,v_2, \dots \in \EE$ are such that $\sum_{n=1}^\infty \|v_n\| \mu_n(X) < \infty$.
\end{theorem}

\begin{proof} Let $\omega$ be a $\EE$-valued measure of bounded variation on $(X,\Sigma)$ that is absolutely continuous with respect $\mu$. If $\EE$ has the Radon-Nikodym property, there exists a Bochner integrable function $f$ on $(X,\Sigma, \mu)$ such that $\omega (S)= \int_S f d\mu$ for every $S\in\Sigma$. Let $A_1, A_2, \dots \in \Sigma$ and $v_1, v_2, \dots \in \EE$ be such that $f \simeq \sum_{n=1}^\infty v_n A_n$ in $(X,\Sigma, \mu)$. Then
\[
\int_S f d\mu=\sum_{n=1}^\infty v_n \mu(S\cap A_n).
\]
If we define 
$\mu_n (S)=\mu(S\cap A_n)$,
then
\[
\omega (S)= \int_S f d\mu= \sum_{n=1}^\infty v_n \mu_n(S).
\]

Now let $\omega$ be an $\EE$-valued measure on $(X,\Sigma)$ of bounded variation that is absolutely continuous with respect $\mu$. If
\[
\omega = \sum_{n=1}^\infty v_n \mu_n
\]
where $\mu_1,\mu_2, \dots$ are positive measures on $(X,\Sigma)$ that is absolutely continuous with respect $\mu$ and $v_1,v_2, \dots \in \EE$ are such that $\sum_{n=1}^\infty \|v_n\| \mu_n(X) < \infty$, we define $f_n=\frac{d\mu_n}{d\mu}$ for $n\in\NN$. Since
\[
\sum_{n=1}^\infty \|v_n f_n\|_1=\sum_{n=1}^\infty \|v_n\| \|f_n\|_1 = \sum_{n=1}^\infty \|v_n\| \mu_n(X) < \infty,
\]
where $\|\cdot\|_1$ denotes the $L^1$-norm with respect to $\mu$, the series $\sum_{n=1}^\infty v_n f_n$ converges to a Bochner integrable function $f$ on $(X,\Sigma, \mu)$ and we have
\[
\omega (S)= \sum_{n=1}^\infty v_n \mu_n(S)=\sum_{n=1}^\infty v_n \int_S f_n d\mu=\int_S \sum_{n=1}^\infty v_nf_n d\mu=\int_S f d\mu,
\]
for every $S\in\Sigma$.
\end{proof}

\begin{corollary}
 A Banach space $\EE$ has the Radon-Nikodym property if and only if every $\EE$-valued measure $\omega$ of bounded variation on any measurable space $(X,\Sigma)$ is of the form $\omega = \sum_{n=1}^\infty v_n \mu_n$, 
where $\mu_1,\mu_2, \dots$ are positive measures on $(X,\Sigma)$ that is absolutely continuous with respect to $|\omega|$ and $v_1,v_2, \dots \in \EE$ are such that $\sum_{n=1}^\infty \|v_n\| \mu_n(X) < \infty$.
\end{corollary}

\section{Extension of transfunctions to vector valued measures}

In this section we consider extensions of transfunctions to vector measures. We start with three technical lemmas that will be used in the proof of the main result.

\begin{lemma}\label{Lem1}
Let $\mu_1, \dots , \mu_n$ be finite positive measures on $(X,\Sigma)$. Then there exists a measure $\mu$ on $(X,\Sigma)$ such that for every $\varepsilon>0$ there is a finite partition 
$\pi \subset \Sigma$ of $X$ such that for $i=1,\dots,n$ we have
\[
\mu_i = \sum_{S \in \pi} \alpha_{i,S} \mu_{S} + \kappa_i
\]
where $\alpha_{i,S}\geq 0$ and $\kappa_1, \dots , \kappa_n$ are positive measures on $(X,\Sigma)$ such that
\begin{equation}
\label{eq:kappa}
\kappa_1(X)  + \dots +  \kappa_n(X)  < \varepsilon.
\end{equation}
In particular, we may define 
$\mu$ to be $\sum_{i=1}^n \mu_i$.
\end{lemma}

\begin{proof} Let $\varepsilon>0$. 
Notice that the the measures 
$\mu_1,\dots, \mu_n$
are absolutely continuous with respect to 
$\mu$ as defined above.
Consider the Radon-Nikodym derivatives $f_i$ of  $\mu_i$, that is,
$\mu_i(B)=\int_B f_i d \mu$
for all $B\in\Sigma$. Since each $f_i$ is a non-negative and integrable with respect to $\mu$, 
there are simple functions 
$
	\sum_{A\in \pi_i}
	\alpha_{i,A}
	A(x),
$
with respect to finite partitions $\pi_i$ of $X$, such that $\alpha_{i,A} \geq 0$, $\sum_{A\in \pi_i} \alpha_{i,A} A(x) \leq f_i$, and 
\begin{equation*}
	\int_X
	\left(
	f_i
	-
	\sum_{A\in \pi_i}
	\alpha_{i,A}A(x)
	\right)
	d\mu
	<
	\frac{\varepsilon}{n}.
\end{equation*} 
Now define the common refinement of the partitions $\pi_i$ to be $\pi$, and define the measures $\kappa_i$ by the equation
\begin{equation*}
	\kappa_i (B)
	=
	\int_B
	\left(
	f_i-
	\sum_{A\in \pi_i}
	\alpha_{i,A}
	A(x)
	\right)
	d\mu
\end{equation*} 
for all $B\in \Sigma$.
Notice that each simple function with respect to $\pi_i$ is also a simple function with respect to $\pi$, that is,
\begin{equation*}
	\sum_{A\in \pi_i}
	\alpha_{i,A}
	A(x)
	=
	\sum_{S\in \pi}
	\alpha_{i,S}
	S(x)
\end{equation*} 
where $\alpha_{i,A} = \alpha_{i,S}$ if 
$S\subseteq A$. Consequently, for every $B\in \Sigma$, we have
\begin{align*}
	\mu_i(B)
	&=
	\int_B
	\sum_{S\in \pi}
	\alpha_{i,S}
	S(x)
	d\mu
	+
	\int_B
	\left(
	f_i-
	\sum_{S\in \pi}
	\alpha_{i,S}
	S(x)
	\right)
	d\mu
	\\
	&=
	\sum_{S\in \pi} 
	\alpha_{i,S}
	\mu_S(B)
	+
	\kappa_{i}(B),
\end{align*}
and the $\kappa_i$'s were constructed to satisfy \eqref{eq:kappa}.
\end{proof}

Let $\Phi: \MS(X,\Sigma_X, \RR) \to \MS(Y,\Sigma_Y, \RR)$ be a transfunction.  We say that $\Phi$ is bounded  if $\|\Phi(\mu)\|\leq C\|\mu\|$ for some $C>0$ and all $\mu\in \mathcal{M}$ and we define \[
\|\Phi \|= \inf \{C: \|\Phi(\mu)\|\leq C\|\mu\| \text{ for all } \mu \in \MS(X,\Sigma_X, \RR)   \},
\]
so that we have $\|\Phi(\mu)\|\leq \|\Phi\|\|\mu\|$ for all $\mu\in \MS(X,\Sigma_X, \RR)$. We say that $\Phi$ is strongly additive if $\Phi(\mu_1+\mu_2)= \Phi(\mu_1)+\Phi(\mu_2)$ for all $\mu_1,\mu_2\in\mathcal{M}$, and we say that $\Phi$ is homogeneous if $\Phi(\alpha \mu)=\alpha \Phi(\mu)$ for any $\alpha > 0$.

\begin{lemma}\label{bound} Let $\Phi: \MS(X,\Sigma_X, \RR^+) \to \MS(Y,\Sigma_Y, \RR^+)$ be a bounded, strongly additive, and homogeneous transfunction and let $\EE$ be a Banach space. Then 
 \[
 \|v_1\Phi\mu_1+ \dots + v_n\Phi\mu_n\| \leq \|\Phi\| \|v_1\mu_1+ \dots + v_n\mu_n\|
 \]
 for all $v_1, \dots , v_n\in \EE$, $\mu_1, \dots , \mu_n\in \MS(X,\Sigma_X, \RR^+)$, and $n\in\NN$.
\end{lemma}

\begin{proof} Let $v_1, \dots , v_n\in \EE$, $\mu_1, \dots , \mu_n\in \MS(X,\Sigma_X, \RR^+)$, and let $\mu=\sum_{k=1}^n \mu_k$. By Lemma \ref{Lem1}, for any $\varepsilon  >0$ there are disjoint sets $S_1, \dots , S_N \in \Sigma$ such that for $i=1,\dots,n$ we have
\[
\mu_i = \sum_{j=1}^N \alpha_{i,j} \mu_{S_j} + \kappa_i,
\]
where $\alpha_{k,j}\geq 0$ and $\kappa_1, \dots , \kappa_n$ are measures such that
\[
\norm{v_1\kappa_1 + \dots + v_n\kappa_n} < \frac\varepsilon{2\|\Phi\|}
\quad \text{and} \quad
\norm{v_1\Phi(\kappa_1) + \dots + v_n\Phi(\kappa_n)} < \frac\varepsilon 2.
\]
Then
\begin{align*}
	\norm{v_1 \Phi(\mu_1) + \dots +v_n \Phi(\mu_n)}
	&=\norm{\sum_{i=1}^n v_i \Phi \left(\sum_{j=1}^N \alpha_{i,j} \mu_{S_j}+ \kappa_i\right)}\\
	&=\norm{\sum_{i=1}^n v_i \left(\sum_{j=1}^N \alpha_{i,j} \Phi (\mu_{S_j})+ \Phi (\kappa_i)\right)}\\
	&=\norm{\sum_{j=1}^N \left( \sum_{i=1}^n \alpha_{i,j}v_i\right) \Phi (\mu_{S_j})+ \sum_{i=1}^n v_i\Phi (\kappa_i)}\\
	&\leq \norm{\sum_{j=1}^N \left( \sum_{i=1}^n \alpha_{i,j}v_i\right) \Phi (\mu_{S_j})}+ \norm{\sum_{i=1}^n v_i\Phi (\kappa_i)}\\
	&\leq \norm{\sum_{j=1}^N \left( \sum_{i=1}^n \alpha_{i,j}v_i\right) \Phi (\mu_{S_j})}+ \frac\varepsilon 2\\
	&\leq \sum_{j=1}^N \norm{\left( \sum_{i=1}^n \alpha_{i,j}v_i\right) \Phi (\mu_{S_j})}+ \frac\varepsilon 2\\
	&\leq \sum_{j=1}^N \norm{ \sum_{i=1}^n \alpha_{i,j}v_i}\norm{ \Phi (\mu_{S_j})}+ \frac\varepsilon 2\\
	&\leq \|\Phi\| \sum_{j=1}^N \norm{ \sum_{i=1}^n \alpha_{i,j}v_i}\norm{ \mu_{S_j}}+ \frac\varepsilon 2.
\end{align*}
Since
\begin{align*}
\norm{v_1\mu_1+ \dots + v_n\mu_n}  &=  \norm{\sum_{i=1}^n v_i\sum_{j=1}^N \alpha_{i,j} \mu_{S_j}+ \sum_{i=1}^n v_i\kappa_i}\\
&\geq \norm{\sum_{i=1}^n v_i\sum_{j=1}^N \alpha_{i,j} \mu_{S_j}}-\frac\varepsilon{2\|\Phi\|}\\
&= \norm{\sum_{j=1}^N \sum_{i=1}^n \alpha_{i,j}v_i \mu_{S_j}}-\frac\varepsilon{2\|\Phi\|}\\ 
&= \sum_{j=1}^N\norm{\sum_{i=1}^n \alpha_{i,j}v_i \mu_{S_j}}-\frac\varepsilon{2\|\Phi\|}\\
&= \sum_{j=1}^N\norm{\sum_{i=1}^n \alpha_{i,j}v_i} \norm{\mu_{S_j}}-\frac\varepsilon{2\|\Phi\|},
\end{align*}
we get
\[
\norm{v_1 \Phi(\mu_1) + \dots +v_n \Phi(\mu_n)} \leq \|\Phi\| \norm{v_1\mu_1+ \dots + v_n\mu_n} +\varepsilon .
\]
Since $\varepsilon $ is an arbitrary positive number, the desired inequality follows.
\end{proof}

\begin{corollary}\label{C5.3}
  Let $\Phi: \MS(X,\Sigma_X, \RR^+) \to \MS(Y,\Sigma_Y, \RR^+)$ be a bounded, strongly additive, and homogeneous transfunction and let $\EE$ be a Banach space. If
  \[
  \sum_{n=1}^\infty v_n\mu_n=0,
  \] 
  for some $v_n\in\EE$ and $\mu_n \in \MS(X,\Sigma_X, \RR^+)$, then
  \[
  \sum_{n=1}^\infty v_n\Phi\mu_n=0.
  \] 
\end{corollary}

\begin{proof} Since
 \[
\left\|\sum_{j=1}^n v_n\Phi \mu_n\right\| \leq \|\Phi\|\left\|\sum_{j=1}^n v_n \mu_n\right\|,
\]
we have
\[
\left\|\sum_{j=1}^\infty v_n\Phi \mu_n\right\| \leq \|\Phi\|\left\|\sum_{j=1}^\infty v_n \mu_n\right\|=0.
\]
\end{proof}

In the next theorem we show that bounded, strongly additive, and homogeneous transfunctions can be extended to vector measures of a special type, namely measures that can be defined as sums of series of positive measures multiplied by elements of a Banach space $\EE$. We will denote this space of measures by $\MS_s(X,\Sigma_X, \EE)$, that is,
 \[
\MS_s(X,\Sigma_X, \EE)=\left\{\sum_{n=1}^\infty v_n\mu_n: \mu_n\in \MS(X,\Sigma_X, \RR^+), v_n\in \EE, \sum_{n=1}^\infty \|v_n\|\|\mu_n\| < \infty \right\}. 
 \]

\begin{theorem}\label{vecext}
 Let $\Phi: \MS(X,\Sigma_X, \RR^+) \to \MS(Y,\Sigma_Y, \RR^+)$ be a bounded, strongly additive, and homogeneous transfunction and let $\EE$ be a Banach space. Then there is a unique bounded linear transfunction $\tilde{\Phi}: \MS_s(X,\Sigma_X, \EE) \to \MS(Y,\Sigma_Y, \EE)$ satisfying $\tilde{\Phi}(v\mu)=v \Phi \mu$ for every $\mu\in \MS(X,\Sigma_X, \RR^+)$ and every $v\in\EE$.
\end{theorem}

\begin{proof} Let $\Phi: \MS(X,\Sigma_X, \RR^+) \to \MS(Y,\Sigma_Y, \RR^+)$ be a bounded, strongly additive, and homogeneous transfunction and let $\EE$ be a Banach space.  

If $\mu = \sum_{n=1}^\infty v_n\mu_n$, for some $v_n\in \EE$ and $\mu_n\in \MS(X,\Sigma_X, \RR^+)$ such that $ \sum_{n=1}^\infty \|v_n\|\|\mu_n\| < \infty$, then we define 
\[
\tilde{\Phi}(\mu) = \sum_{n=1}^\infty v_n \Phi\mu_n. 
\] 
Since
\[
\sum_{n=1}^\infty \norm{v_n \Phi\mu_n} 
\leq \sum_{n=1}^\infty \|v_n\| \|\Phi\mu_n\|
\leq \|\Phi\| \sum_{n=1}^\infty \|v_n\| \|\mu_n\| <\infty,
\]
the series converges. Moreover, if 
$
\sum_{n=1}^\infty v_n\mu_n= \sum_{n=1}^\infty w_n\kappa_n,
$
for some $v_n, w_n\in \EE$ and $\mu_n, \kappa_n\in \MS(X,\Sigma_X, \RR^+)$ such that 
\[
\sum_{n=1}^\infty \|v_n\|\|\mu_n\| < \infty \quad \text{and} \quad \sum_{n=1}^\infty \|w_n\|\|\kappa_n\| < \infty,
\] 
then
\[
v_1\mu_1-w_1\kappa_1+v_2\mu_2-w_2\kappa_2+ \dots =0.
\]
By Corollary \ref{C5.3},
\[
v_1\Phi\mu_1-w_1\Phi\kappa_1+v_2\Phi\mu_2-w_2\Phi\kappa_2+ \dots =0,
\]
so $\sum_{n=1}^\infty v_n\Phi\mu_n= \sum_{n=1}^\infty w_n\Phi\kappa_n$. This shows that the extension is well-defined.

Clearly, $\tilde{\Phi}$ is a linear transfunction from $\MS_s(X,\Sigma_X, \EE)$ to $\MS(Y,\Sigma_Y, \EE)$. Since, by Lemma \ref{bound}, for every $n\in\NN$ we have 
\[
 \|v_1\Phi\mu_1+ \dots + v_n\Phi\mu_n\| \leq \|\Phi\| \|v_1\mu_1+ \dots + v_n\mu_n\|,
 \]
we have $\|\tilde{\Phi}(\mu)\| \leq \|\Phi\|\|\mu\|$, so $\tilde{\Phi}$ is bounded. 

Now let $\Psi: \MS_s(X,\Sigma_X, \EE) \to \MS(Y,\Sigma_Y, \EE)$ be a bounded linear transfunction satisfying $\Psi(v\mu)=v \Psi \mu$ for every $\mu \in \MS(X,\Sigma_X, \RR^+)$ and every $v\in\EE$. If $\mu = \sum_{n=1}^\infty v_n\mu_n$, for some $v_n\in \EE$ and $\mu_n\in \MS(X,\Sigma_X, \RR^+)$ such that $\ds \sum_{n=1}^\infty \|v_n\|\|\mu_n\| < \infty$, then for every $n\in\NN$ we have
\[
\Psi(v_1\mu_1+ \dots + v_n\mu_n)= \tilde{\Phi}(v_1\mu_1+ \dots + v_n\mu_n)
\] 
and consequently $\tilde{\Phi}=\Psi$ by continuity.
\end{proof}

From the above theorem and Theorem \ref{Th4.1} we obtain the following result.

\begin{corollary}\label{Cor3.8}
 Let $\EE$ be a Banach space with the Radon-Nikodym property. For every bounded, strongly additive, and homogeneous $\Phi: \MS(X,\Sigma_X, \RR^+) \to \MS(Y,\Sigma_Y, \RR^+)$ there is a unique bounded linear transfunction $\tilde{\Phi}: \MS(X,\Sigma_X, \EE) \to \MS(Y,\Sigma_Y, \EE)$ satisfying $\tilde{\Phi}(v\mu)=v \Phi \mu$ for every $\mu\in \MS(X,\Sigma_X, \RR^+)$ and every $v\in\EE$.  
\end{corollary}

\end{document}